\documentclass[12pt,leqno]{amsart}
\usepackage{amsmath}
\usepackage{amssymb}
\def\Dbar{\leavevmode\lower.6ex\hbox to 0pt{\hskip-.03ex\accent"16\hss}D}

\def\underset#1#2{{\mathrel{\mathop {{}_{} {#2}}\limits_{{#1}_{}}}}}
\def\upplim_#1{\underset{#1}{\overline\lim}\;}
\def\lowlim_#1{\underset{#1}{\underline\lim}\;}
\newtheorem{claim}[equation]{{\it Claim}\rm }

\newtheorem{lemma}[equation]{Lemma}

\newtheorem{theorem}[equation]{Theorem}

\newcommand{\C}{{\mathbb{C}}}

\renewcommand{\P}{{\mathbb{P}}}

\newcommand{\R}{{\mathbf{R}}}

\newcommand{\rank}{\mathrm{rank}}

\numberwithin{equation}{section}

\title[Second main theorems with weighted counting functions...]{Second main theorems with weighted counting functions and its applications} 

\date { }
\author[Pham Duc Thoan, Nguyen Hai Nam and Nguyen Van An]{Pham Duc Thoan$^{1}$, Nguyen Hai Nam$^1$ and Nguyen Van An$^2$}

\address{$^{1}$ Department of Mathematics, National University of Civil Engineering\\ 55 Giai Phong str., Hanoi, Vietnam}
\email{thoanpd@nuce.edu.vn, namnh211@gmail.com}

\address{$^2$ Division of Mathematics, Banking Academy,\\
12-Chua Boc, Dong Da, Hanoi, Vietnam}
\email{an0883@gmail.com}

\begin{document}

\begin{abstract} 
The purpose of this article has two fold. The first is to generalize some recent second main theorems for the mappings and moving hyperplanes of $\P^n(\C)$ to the case where the counting functions are truncated multiplicity (by level $n$) and have different weights. As its application, the second purpose of this article is to generalize and improve some algebraic dependence theorems for meromorphic mappings having the same inverse images of some moving hyperplanes to the case  where the moving hyperplanes involve the assumption with different roles.
\end{abstract}

\def\thefootnote{\empty}
\footnotetext{
The research is funded by National University of Civil Engineering (NUCE) under grant number 203-2018/KHXD-T\Dbar.}

\maketitle

\section{Introduction}
The theory on second main theorem for meromorphic mappings into projective spaces with moving hyperplanes was started studied by W. Stoll, M. Ru \cite{RS} and M. Shirosaki in 1990's \cite{S1,S2}. In that time, almost all given second main therems do not have the truncation level for the counting functions of the inverse image of moving hyperplanes. In some recent years, this theory have been studied very intesively with many results established. To state some of them, we recall the following notation.

Let $a_1,\dots,a_q$ $(q \geq n+1)$ be $q$ meromorphic mappings of $\C^m$ into the dual space $\P^n(\C)^*$ with reduced representations $a_i = (a_{i0}: \dots : a_{in})\ (1\le i \le q).$ We say that $a_1,\dots,a_q$ are located in general position if $\det (a_{i_kl}) \not \equiv 0$ for any $1\le i_0<i_1<\cdots <i_n\le q.$ Let $\mathcal {M}_m$ be the field of all meromorphic functions on $\C^m$. Denote by $\mathcal {R}(\{a_i\}_{i=1}^q) \subset \mathcal {M}_m$ the smallest subfield which contains $\C$ and all $ \frac {a_{ik}}{a_{il}}\text { with } a_{il}\not\equiv 0.$ 

For the case of nondegenerate meromorphic mappings of $\C^m$ into $\P^n(\C)$ intersecting moving hyperplanes, the first second main theorem with truncated (to level $n$) counting functions was given by Ru \cite{MR} for the case $m=1$ and reproved for general case by Thai-Quang \cite{TQ05}. For the case of degenerate meromorphic mappings, in \cite{RW}, Ru and Wang gave a second main theorem for moving hyperplanes with counting function truncated to level $n$. And then, the result of Ru-Wang was improved by Thai-Quang \cite{TQ08} and Quang-An \cite{QA}. In 2016, S. D. Quang \cite{Q14} improved and extended these results to the following.
 
\vskip0.2cm
\noindent{\bf Theorem A }\cite[Theorem 1.1]{Q14} {\it Let $f: \mathbb C^m\to\mathbb P^n(\mathbb C)$ be a meromorphic mapping. Let $\{a_j\}_{j=1}^q$\ $(q\geq 2n-k+2)$ be meromorphic mappings of $\mathbb C^m$ into $\mathbb P^n(\mathbb C)^*$ in general position such that $(f,a_j)\not\equiv0$\ $(1\leq j\leq q)$, where $\rank_{R\{a_j\}}(f)=k+1$. Then the following assertion holds:\\
(a) $\big|\big|\ \displaystyle\frac{q}{2n-k+2}T_f(r)\leq\sum_{i=1}^q N^{[k]}_{(f_i,a)}(r)+o(T_f(r))+O(\max\limits_{1\leq i\leq q}T_{a_i}(r)),$\\
(b) $\big|\big|\ \displaystyle\frac{q-(n+2k-1)}{n+k+1}T_f(r)\leq\sum_{i=1}^q N^{[k]}_{(f_i,a)}(r)+o(T_f(r))+O(\max\limits_{1\leq i\leq q}T_{a_i}(r)).$}
\vskip0.2cm
Here, by the notation $``|| \ P"$  we mean the assertion $P$ holds for all $r \in [0,\infty)$ outside a Borel subset $E$ of the interval $[0,\infty)$ with $\int_E dr<\infty$. 

Recently, S. D. Quang \cite{Q18} has improved his result to the following.

\vskip0.2cm
\noindent{\bf Theorem B }\cite[Theorem 1.1 (a)]{Q18} {\it Let $f: \mathbb C^m\to\mathbb P^n(\mathbb C)$ be a meromorphic mapping. Let $\{a_j\}_{j=1}^q$\ $(q\geq 2n-k+2)$ be meromorphic mappings of $\mathbb C^m$ into $\mathbb P^n(\mathbb C)^*$ in general position such that $(f,a_j)\not\equiv0$\ $(1\leq j\leq q)$, where $\rank_{R\{a_j\}}(f)=k+1$. Then we have
$$ \ || \ \dfrac {q-(n-k)}{n+2}T_f(r) \le \sum_{i=1}^q N_{(f,a_i)}^{[k]}(r) + o(T_f(r)) + O(\max_{1\le i \le q}T_{a_i}(r)). $$}

\vskip0.2cm
In another direction, in 2015, S. D. Quang \cite{Q15} initially introduced the second main theorem with weighted counting functions. He has generalized partially the above results (the assertion (a) of Theorem A) to the case where each counting function has a different weight. His result is stated as follows.

\vskip0.2cm
\noindent{\bf Theorem C} \cite[Theorem 1.1]{Q15} {\it Let $f :\C^m \to \P^n(\C)$ be a meromorphic mapping. Let $\{a_i\}_{i=1}^q \ (q\ge 2n-k+2)$ be meromorphic mappings of $\C^m$ into $\P^n(\C)^*$ in general position such that $(f,a_i)\not\equiv 0\ (1\le i\le q).$ Assume that $k+1=\rank_{\mathcal R\{a_i\}}(f)$. Let $\lambda_1,\ldots ,\lambda_q$ be $q$ positive number with $(2n-k+2)\max_{1\le i\le q}\lambda_i\le\sum_{i=1}^q\lambda_i$. Then the following assertions hold:
\begin{align*}
|| \ \dfrac {\sum_{i=1}^q\lambda_i}{2n-k+2}T_f(r) \le \sum_{i=1}^q \lambda_iN_{(f,a_i)}^{[k]}(r) + o(T_f(r)) + O(\max_{1\le i \le q}T_{a_i}(r)).
\end{align*}}

\vskip0.2cm
Our first aim in this paper is to give a complete generalization of these above results in this direction.  Namely, we will generalize Theorem B to the following.

\begin{theorem}\label{1.1} Let $f: \mathbb C^m\to\mathbb P^n(\mathbb C)$ be a meromorphic mapping. Let $\{a_j\}_{j=1}^q$\ $(q\geq 2n-k+2)$ be meromorphic mappings of $\mathbb C^m$ into $\mathbb P^n(\mathbb C)^*$ in general position such that $(f,a_j)\not\equiv0$\ $(1\leq j\leq q)$. Assume that $k+1=\rank_{\mathcal R\{a_i\}}(f)$. Let $\lambda_1,...,\lambda_q$ be $q$ positive numbers with $(2n-k+2)\max_{1\le i\le q}\lambda_i\le \sum_{i=1}^q\lambda_i$. Then for every positive number $\eta\in [\max_{1\le i\le q}\lambda_i,\frac{\sum_{i=1}^q\lambda_i}{2n-k+2}]$, we have
$$\big|\big|\ \frac{\sum_{j=1}^q\lambda_j-(n-k)\eta}{n+2}T_f(r)\leq\sum_{j=1}^q\lambda_{j}N^{[k]}_{(f,a_{j})}(r)+o(T_f(r))+O(\max\limits_{1\leq i\leq q}T_{a_i}(r)).$$
\end{theorem}
\noindent
\textit{\underline{Remark:}} 1) Letting $\lambda_1=\cdots =\lambda_q=1$ and $\eta=1$, from Theorem \ref{1.1}, we get Theorem B.

2) Letting $\eta=\frac{\sum_{i=1}^q\lambda_i}{2n-k+2}$, we will get again Theorem C.

\vskip0.2cm
In the last part, we will use the above second main theorem to study algebraic dependence of meromorphic mappings sharing moving hyperplanes regardless of multiplicities. To state our result, we recall the following notation, due to \cite{C,Q15,MR,TDQ} and \cite{Q16}.

 Let $f_t: \mathbb C^m\rightarrow \mathbb P^n(\mathbb C) \ (1\leqslant t \leqslant \lambda)$ be meromorphic mappings with reduced representations
 $f_t:=(f_{t0}:\cdots :f_{tn}).$ Let $a_j: \mathbb C^m\rightarrow \mathbb P^n(\mathbb C)^* \ (1\leqslant j\leqslant q)$ be moving targets located in
general position with reduced representations $a_j:=(a_{j0}:\cdots :a_{jn}).$ Assume that $(f_t,a_j):= \sum_{i=0}^nf_{ti}a_{ji}\ne 0$ for each  
$1\le t\le\lambda ,\ 1\le j\le q$ and $(f_1,a_j)^{-1}\{0\}=\cdots =(f_{\lambda},a_j)^{-1}\{0\}$. Put $A_j=(f_1,a_j)^{-1}\{0\}$ for each $1\leqslant j\leqslant q.$
Assume that every analytic set $A_j$ has the irriducible decomposition as follows $A_j=\cup_{i=1}^{t_j}A_{ji}(1\leqslant t_j\leqslant\infty)$.
Set $A=\cup_{A_{ji}\not\equiv A_{kl}}\{A_{ji} \cap A_{kl}\}$ 
with $1\le i\le t_j,1\le l\le t_k, 1\le j, k \le q$.

Denote by $T[N+1,q]$ the set of all injective maps from $\{1,\cdots N+1\}$ to $\{1,\cdots ,q\}.$ 
For  each $z\in \mathbb C^n\setminus\{\cup_{\beta\in T[N+1,q]}\{z|a_{\beta (1)}(z)\wedge\cdots\wedge a_{\beta (N+1)}(z)=0\} \cup A\cup\cup_{i=1}^{\lambda}I(f_i)\},$ we define $\rho (z)=\sharp\{j|z\in A_j\}$. 
Then $\rho (z)\le N.$ Indeed, suppose that $z\in A_j$ for each $0\le j\le N.$  Then $\sum_{i=0}^{N}f_{1i}(z)\cdot a_{ji}(z)=0$ for each $0\le j\le N.$ 
Since $a_{\beta (1)}(z)\wedge\cdots\wedge a_{\beta (N+1)}(z)\ne 0,$ it implies that $f_{1i}(z)=0$ for each $0\le i\le N.$ This means that 
$z\in I(f_1).$ This is impossible.

For any positive number $r>0,$ define $\rho (r)=\sup\{\rho (z)| |z|\le r\},$ where the supremum is taken over all
$z\in \mathbb C^n\setminus\{\cup_{\beta\in T[N+1,q]}\{z|a_{\beta (1)}(z)\wedge\cdots\wedge a_{\beta (N+1)}(z)=0\} \cup A\cup\cup_{i=1}^{\lambda}I(f_i)\}.$ 
Then $\rho (r)$ is a decreasing function. Let $$ d:=\lim_{r\rightarrow +\infty}\rho (r).$$
Then $d\le N.$  If for each $i \ne j, \ \dim\{A_i\cap A_j\}\le n-2,$ then $d=1.$

In 2001, M. Ru \cite{MR} proved the following theorem. 

\vskip0.2cm
\noindent
\textbf{Theorem B}\ (see \cite[Theorem 1]{MR}){\it Let $f_1,\cdots ,f_{\lambda}: \C^m\rightarrow \P^n(\C)$ $(\lambda \ge 2)$ be nonconstant meromorphic mappings. Let $a_i:\C^m \rightarrow \P^n(\C)^*\ (1\le i\le q)$ be slowly moving hyperplanes in general position. Assume that $(f_i,a_j)\not\equiv 0$ and $(f_1,a_j)^{-1}\{0\}=\cdots =(f_{\lambda},a_j)^{-1}\{0\}$ for each $1\le i\le \lambda, 1\le j\le q$. Denote $A_j=(f_1,a_j)^{-1}(\{0\})$. Let $l$ be a positive integer with $2\le l\le\lambda$. Assume that for each $z\in A_j\ (1\le j\le q)$ and for any $1\le i_1<\cdots <i_{l_j}<q,\ f_{i_1}(z)\wedge\cdots\wedge f_{i_l}(z)=0$.
If $q>\dfrac{d\lambda n^2(2n+1)}{\lambda -l +1},$ then $f_1,\cdots ,f_{\lambda} $ are algebraically 
dependent over $\C,$ i.e.,  $f_1\wedge\cdots\wedge f_{\lambda}\equiv 0$ on $\C^m.$}

\vskip0.2cm 
After that, the result of M. Ru has been improved and extended by P. D. Thoan - P. V. Duc and S. D. Quang in \cite{Q12,Q15,TD,TDQ} when the number of moving hyperplanes is reduced. In 2015, L. N. Quynh \cite{Q16} proposed a new technique, by which she studied the algebraic dependence of meromorphic mappings sharing different family of moving hyperplanes regardless of multiplicities and obtained the results which are much more general and stronger than previous results. Inspired of the technique of Quynh, in this paper we consider the case where the number $l$ in the above theorem may varies dependently on the moving hyperplanes. Namely, we will prove the following.

\begin{theorem}\label{1.2}
Let $f_1,\cdots ,f_{\lambda}: \C^m\rightarrow \P^n(\C)$ $(\lambda \ge 2)$ be nonconstant meromorphic mappings. Let $a_i:\C^m \rightarrow \P^n(\C)^*\ (1\le i\le q)$ be slowly moving hyperplanes in general position. Assume that $(f_i,a_j)\not\equiv 0$ and $(f_1,a_j)^{-1}\{0\}=\cdots =(f_{\lambda},a_j)^{-1}\{0\}$ for each $1\le i\le \lambda, 1\le j\le q$. Denote $A_j=(f_1,a_j)^{-1}(\{0\})$. Let $l_1,\ldots ,l_q$ be $q$ positive integers with $2\le l_i\le\lambda$. Assume that for each $z\in A_j\ (1\le j\le q)$ and for any $1\le i_1<\cdots <i_{l_j}<q,\ f_{i_1}(z)\wedge\cdots\wedge f_{i_{l_j}}(z)=0$.
If $q>\frac{d\lambda k(2n-k+2)-d\lambda (k-1)+\sum_{j=1}^ql_j}{\lambda +1},$ then $f_1,\cdots ,f_{\lambda} $ are algebraically 
dependent over $\C,$ i.e.,  $f_1\wedge\cdots\wedge f_{\lambda}\equiv 0$ on $\C^m.$
\end{theorem}

In particular, if let $\lambda=l=2$ and $d=1$ then the condition of the above theorem is fulfilled with $q>2n^2+2n+2$. Therefore, we will get the uniqueness theorem for meromorphic mappings (which may be degenerate) sharing $q>2n^2+2n+2$ slowly moving hyperplanes in genral position without multiplicity. This conclusion had been proved by L. N. Quynh (see \cite[Corollary 1.4]{Q16}), and independently by H.Z. Cao (see \cite[Corollary 3]{C}).
 
\section{Basic notions and auxiliary results from Nevanlinna theory}

\noindent
\textbf{(a)} Basic notions.

Throughout this paper, we use the standart notation on Nevanlina theory due to \cite{C,Q12,Q15} and \cite{Q16}. For a meromorphic mapping $f:\C^m\to\P^n(C)$, we denote by $T_f(r)$ its characteristic funtion. For a diviosr $D$ on $\C^m$, we denote by $N^{[k]}(r,D)$ its counting function trucated to level $k$. We mean by a moving hyperplanes a meromorphic mapping $a:\C^m\to\P^n(\C)^*$. Such $a$ is said to be slow with respect to $f$ if $||\ T_a(r)=o(T_f(r))$.  Let $\varphi$ be a meromorphic funtion on $\C^m$. We denote by $\nu_\varphi$ its divisor and denote by $N_\varphi (r)$ the counting function of its zeros divisor. 

We assume that thoughout this paper, the homogeneous coordinates of $\P^n(\C)$ is chosen so that for each given meromorphic mapping $a=(a_0:\cdots :a_n)$ of $\C^m$ into $\P^n(\C)^*$ then $a_{0}\not\equiv 0$. We set
$$ \tilde a_i=\dfrac{a_i}{a_0}\text{ and }\tilde a=(\tilde a_0:\tilde a_1:\cdots:\tilde a_n).$$
Supposing that $f$ has a reduced representation $f=(f_0:\cdots :f_n).$  We put $(f,a):=\sum_{i=0}^{n}f_ia_{i}$ and $(f,\tilde a):=\sum_{i=0}^{n}f_i\tilde a_{i}.$

Let $\{a_i\}_{i=1}^q$ be $q$ meromorphic mappings of $\C^m$ into $\P^n(\C)^*$  with reduced representations $a_i=(a_{i0}:\cdots :a_{in})\ (1\le i\le q).$ We denote by  $\mathcal R(\{a_i\})$ (for brevity we will write $\mathcal R$ if there is no confusion) the smallest subfield of $\mathcal M$ which contains $\C$ and all ${a_{i_j}}/{a_{i_k}}$ with $a_{i_k}\not\equiv 0.$


\noindent
\textbf{(b)} Theorems for general position.

\begin{theorem}[{The First Main Theorem for general position \cite[p. 326]{St2}}]
Let $f_i: \C^m\rightarrow \P^n(\C),\ 1\le i\le k$ be meromorphic mappings located in general position. Assume that $1\le k \le n.$ Then
$$N_{\mu_{f_1\wedge\cdots\wedge f_{\lambda}}}(r)+m(r,f_1\wedge\cdots\wedge f_{\lambda})\le \sum_{1\le i\le\lambda}T_{f_i}(r)+O(1).$$
\end{theorem}
Here, by $\mu_{f_1\wedge\cdots\wedge f_{\lambda}}$ we denote the divisor associated to $f_1\wedge\cdots\wedge f_{\lambda}$.

Let $V$ be a complex vector space of dimension $N\ge 1.$ The vectors  $\{v_1,\cdots ,v_k\}$ are said to be in general position if for each selection of integers $1\le i_1<\cdots <i_p\le k$ with $p\le N,$ then  $v_{i_1}\wedge\cdots\wedge v_{i_p}\ne 0$. The vectors  $\{v_1,\cdots ,v_k\}$ are said to be in 
special position if they are not in general position. Take $1\le p \le k.$ Then $\{v_1,\cdots ,v_k\}$ are said to be in 
$p$-special position if for each selection of integers $1\le i_1<\cdots <i_p\le k,$ the vectors $v_{i_1},\cdots, v_{i_p}$ are in special position.

\begin{theorem}[{The Second Main Theorem for general position \cite[Theorem 2.1, p.320]{St2}}]
Let $M$ be a connected complex manifold of dimension $m.$ Let $A$ be a pure $(m-1)$-dimensional analytic subset of $M.$ Let $V$ be a complex vector space of dimension $n+1>1.$ Let $p$ and $k$ be integers with $1\le p\le k\le n+1.$ Let $f_i:M\rightarrow P(V),1\le i\le k,$ be meromorphic mappings. Assume that $f_1,...,f_k$ are in general position. Also assume that $f_1,...,f_k$ are in $p$-special position on $A.$ Then we have
$$\mu_{f_1\wedge\cdots\wedge f_k}\ge (k-p+1)\nu_{A}.$$
\end{theorem}

\section{The proof of Theorem \ref{1.1}}

In order to prove Theorem \ref{1.1}, we need the following lemma due to Si Duc Quang \cite{Q18}.
\begin{lemma}[{see \cite[Theorem 1.1, equation (3.9)]{Q18}}]\label{Quang}
{\it Let $f: \mathbb C^m\to\mathbb P^n(\mathbb C)$ be a meromorphic mapping. Let $\{a_j\}_{j=1}^{2n-k+2}$ be meromorphic mappings of $\mathbb C^m$ into $\mathbb P^n(\mathbb C)^*$ in general position such that $(f,a_j)\not\equiv 0$\ $(1\leq j\leq 2n-k+2)$, where $\rank_{R\{a_j\}}(f)=k+1$. Then there exists a subset $J\subset\{1,...,2n-k+2\}$ with $|J|= n+2$ satisfying
$$ \ || \ T_f(r) \le \sum_{j\in J}N_{(f,a_j)}^{[k]}(r) + o(T_f(r)) + O(\max_{1\le j \le 2n-k+2}T_{a_j}(r)).$$}
\end{lemma}
Actually, in the proof of Theorem B, firstly S. D. Quang proved this lemma, but he did not separate this lemma from the proof of Theorem B (see equation (3.9) in \cite{Q18}). We also note that, Quang proved that the existing subset $J$ in the above lemma satisfies $|J|\le n+2$. Therefore, it of course will be hold for some $J$ with $|J|=n+2$. 
\begin{proof}[Proof of Theorem \ref{1.1}]

We denote by $\mathcal I$ the set of all permutations of $q-$tuple $(1,\ldots,q)$. For each element $I=(i_1,\ldots,i_q)\in\mathcal I$, we set
$$N_I=\{r\in\R^+;N^{[k]}_{(f,a_{i_1})}(r)\le\cdots\le N^{[k]}_{(f,a_{i_q})}(r)\}.$$

Fix a permutation $I=(i_1,\ldots,i_q)\in\mathcal I$. Applying Lemma \ref{Quang}, there exists a subset $J\in\{1,...,2n-k+2\}$ with $|J|= n+2$ such that
\begin{align*}
||\ T_f(r)&\le\sum_{j\in J}N^{[k]}_{(f,a_{i_j})}+o(T_f(r))+O(\max_{1\le i \le q}T_{a_i}(r)),
\end{align*}
Put $J_1=\{1,...,2n-k+2\}\setminus J$ then 
$$|J_1|=(2n-k+2)-|J|= n-k.$$
By the assumption of the theorem, we have $\sum_{j\in J_1}\lambda_{i_j}-|J_1|\eta\leq 0$ and $\sum_{j=1}^q\lambda_j-|J_1|\eta>0$. Hence 
\begin{equation}\label{Eq2}
\begin{aligned}
\bigl |\bigl |\ (\sum_{j=1}^q\lambda_j-|J_1|\eta)T_f(r)&\leq (\sum_{j\not\in J_1}\lambda_{i_j})\sum_{l\in J} N^{[k]}_{(f,a_{i_l})}(r)+(\sum_{j\in J_1}\lambda_{i_j}-|J_1|\eta)\sum_{l\in J} N^{[k]}_{(f,a_{i_l})}(r)\\
&+o(T_f(r))+O(\max\limits_{1\leq i\leq q}T_{a_i}(r))\\
&\leq|J|\Big(\sum_{l\in J}\frac{\sum\limits_{j\not\in J_1}\lambda_{i_j}}{|J|} N^{[k]}_{(f,a_{i_l})}(r)\Big)+o(T_f(r))+O(\max\limits_{1\leq i\leq q}T_{a_i}(r))\\
&=|J|\left(\sum_{l\in J}\lambda_{i_l}N^{[k]}_{(f,a_{i_l})}(r)+\sum_{l\in J}\Big(\frac{\sum\limits_{j\not\in J_1}\lambda_{i_j}}{|J|}-\lambda_{i_l}\Big) N^{[k]}_{(f,a_{i_l})}(r)\right)\\
&+o(T_f(r))+O(\max\limits_{1\leq i\leq q}T_{a_i}(r)).
\end{aligned}
\end{equation}
Also by the assumption, for each $l\in J$ we have
\begin{equation*}
\begin{aligned}
\sum_{j\not\in J_1}\lambda_{i_j}-|J|\lambda_{i_l}&=\sum_{j=1}^q\lambda_j-\sum_{j\in J_1}\lambda_{i_j}-|J|\lambda_{i_l}\\
&\geq\sum_{j=1}^q\lambda_j-(|J_1|+|J|)\eta\\
&=\sum_{j=1}^q\lambda_j-(2n-k+2)\eta\geq0.
\end{aligned}
\end{equation*}
Then, from (\ref{Eq2}), for all $r\in N_I$, we have
\begin{equation*}
\begin{aligned}
\bigl |\bigl | \left(\sum_{j=1}^q\lambda_j-|J_1|\eta\right)T_f(r)&\leq|J|\left(\sum_{l\in J}\lambda_{i_l}N^{[k]}_{(f,a_{i_l})}(r)+\sum_{l=2n-k+3}^q\lambda_{i_l} N^{[k]}_{(f,a_{i_{2n-k+2}})}(r)\right)\\
&+o(T_f(r))+O(\max\limits_{1\leq i\leq q}T_{a_i}(r))\\
&\leq|J|\sum_{j=1}^q\lambda_{j}N^{[k]}_{(f,a_{j})}(r)+o(T_f(r))+O(\max\limits_{1\leq i\leq q}T_{a_i}(r)).\\
\end{aligned}
\end{equation*}
Hence, for $r\in N_I,$ we have
 \begin{equation*}
\begin{aligned}
\bigl |\bigl | \frac{\sum_{j=1}^q\lambda_j-(n-k)\eta}{n+2}T_f(r)\leq\sum_{j=1}^q\lambda_{j}N^{[k]}_{(f,a_{j})}(r)+o(T_f(r))+O(\max\limits_{1\leq i\leq q}T_{a_i}(r)).
\end{aligned}
\end{equation*}  
The theorem is proved.
\end{proof}

\section{The proof of Theorem \ref{1.2}}

In order to prove Theorem \ref{1.2}, we need the following lemma due to Quynh \cite[Claim 3.3]{Q16} (see also \cite[Claim 3.1]{TD}).
\begin{lemma}\label{4.1}
Let $h_i:\mathbb C^m\to\mathbb P^n(\mathbb C)\ (1\leq i\leq p\leq  n+1)$ be meromorphic mappings with reduced representations $h_i:=(h_{i0}:\cdots:h_{in})$. Let $a_i:\mathbb C^m\to\mathbb P^n(\mathbb C)\ (1\leq i\leq p\leq  n+1)$ be moving hyperplanes with  reduced representations $a_i:=(a_{i0}:\cdots:a_{in})$. Put $\tilde{h_i}:=((h_i,a_1):\cdots:(h_i,a_{n=1})).$ Assume that $a_1,\cdots,a_{n+1}$ are located in general position such that $(h_i,a_j)\not\equiv0\ (1\leq i\leq p, 1\leq j\leq n+1)$. Let $S$ be a pure $(n-1)-$ dimensional analytic subset of $\mathbb C^m$ such that $S\not\subset(a_1\wedge\cdots\wedge a_{n+1})^{-1}\{0\}.$ Then $h_1\wedge\cdots\wedge h_{p}\equiv0$ on $S$ if and only if $\tilde{h_1}\wedge\cdots\wedge \tilde{h_p}\equiv0$ on $S$.
\end{lemma}

\begin{proof}[Proof of Theorem \ref{1.2}]
It suffices to prove the theorem in the case of $\lambda\leq n+1.$ Suppose that $f_1\wedge\cdots\wedge f_{\lambda}\not\equiv0.$

We set $\mathcal A=\cup_{i=1}^qA_i$, and denote by $\mathcal S$ the singular part of $\mathcal A$. For an ordered set of $\lambda$ distinc indices $I=\{j_1,\cdots,j_{\lambda}\}\subset\{1,\ldots ,q\}$, we put $I^c=\{1,\cdots,q\}\setminus I$ and 
$$B_I=\left(\begin{array}{cccc}
(f_1,a_{j_1})&\cdots&(f_{\lambda},a_{j_1})\\
(f_1,a_{j_2})&\cdots&(f_{\lambda},a_{j_2})\\
\vdots&\vdots&\vdots\\
(f_1,a_{j_{\lambda}})&\cdots&(f_{\lambda},a_{j_{\lambda}})
\end{array}\right).$$
 Take a positive number $r_0>1$ such that $\rho (r)=d$ for all $r>r_0$. We now prove the following claim.
\begin{claim}
If $B_I$ is nondegenerate, i.e., $\det B_I\not\equiv0$ then 
$$
d\sum_{i\in I}(\min_{1\leq v\leq \lambda}\nu_{(f_v,a_i)}(z)-\min\{1,\nu_{(f_1,a_i)}\}(z))+\sum_{i=1}^{q}(\lambda-l_i+1)\min\{1,\nu_{(f_1,a_i)}(z)\}\leq d\nu_{\tilde{f_1}\wedge\cdots\wedge\tilde{f_{\lambda}}}(z)
$$
for all $z\in\C^m\setminus(\mathcal S\cup\bigcup_{i=1}^{\lambda}I(f_i)\cup(a_{i_1}\wedge\cdots\wedge a_{j_{\lambda}}))^{-1}(0)$ with $||z||>r_0$, where $\tilde{f_i}:=((f_i,a_{j_1}):\cdots:(f_i,a_{j_{\lambda}}))$.
\end{claim}
Indeed, fix a point $z_0\in \C^m\setminus(\mathcal S\cup\bigcup_{i=1}^{\lambda}I(f_i)\cup(a_{i_1}\wedge\cdots\wedge a_{j_{\lambda}}))^{-1}(0)$ with $||z_0||>r_0$. It suffices for us to prove the inequality of the claim for $z_0\in\mathcal A$. We may assume that:
\begin{itemize}
\item there are $t$ indices in $I$, for intance they are $j_0,...,j_t$ such that $z_0\in\bigcup_{i=1}^tA_{j_i}$ and $z_0\not\in A_{j_i}$ for all $t<i\le\lambda$ ($t$ may be $0$),
\item there are $s$ indices in $I^c$, for intance they are $k_0,...,k_s$ such that $z_0\in\bigcup_{i=1}^sA_{k_i}$ and $z_0\not\in A_k$ for all $k\in I^c\setminus\{k_1,\ldots ,k_s\}$ ($s$ may be $0$), where $s+t\le d$.
\end{itemize}

Let $\Gamma$ be an irriducible analytic subset of $\mathcal A$ containing $z_0.$ Suppose that $U$ is an open neighbourhood of $z_0$ in $\mathbb C^m$ such that $U\cap (\mathcal A\setminus \Gamma)=\emptyset$. Choose holomorphic function $h_i$ on a neighbourhood $U'\subset U$ of $z_0$ such that  
$\nu_{h_i}=\min_{1\le v\le \lambda}\{\nu_{(f_v,a_{j_i})}\}$ if $z\in \Gamma$ and $\nu_{h}=0$ if $z\not\in \Gamma$ for each $1\le i\le t$. Then $(f_v,a_{j_i})=a_{vi}h_i \ (1\le i\le t),$ where $a_{vi}$ are holomorphic functions. Hence, we have
$$ \det B_I=h_1\cdots h_t\cdot\det
 \left(\begin{array}{cccc}
a_{11}&\cdots&a_{\lambda 1}\\
\vdots&\vdots&\vdots\\
a_{1t}&\cdots&a_{\lambda t}\\
(f_1,a_{j_{t+1}})&\cdots&(f_{\lambda},a_{j_{t+1}})\\
\vdots&\vdots&\vdots\\
(f_1,a_{j_{\lambda}})&\cdots&(f_{\lambda},a_{j_{\lambda}})
\end{array}\right).$$
This implies that
\begin{align}\label{4.2}
d\nu_{\tilde{f_1}\wedge\cdots\wedge\tilde{f_{\lambda}}}(z_0)=d\nu_{\det B_I}(z_0)\ge\sum_{i=1}^td\nu_{h_i}(z_0)+d\nu_{g_{t+1}\wedge\cdots\wedge g_{\lambda}}(z_0),
\end{align}
where $g_i=((f_1,a_{j_i}),\cdots,(f_{\lambda},a_{j_i}))$ for every $1\le i\le\lambda$. Let $l=\min\{l_{j_1},..,l_{j_t},k_1,...,k_s\}$. By the assumption, on the analytic set $\Gamma$, we have
$$ \rank\{g_{t+1},\ldots, g_{\lambda}\}=\rank\{g_1,\ldots ,g_\lambda\}=\rank\{\tilde f_1,...,\tilde f_\lambda\}\le l. $$
 By using The Second Main Theorem for general position \cite[Theorem 2.1, p.320]{St2}, we have 
$$\nu_{g_{t+1}\wedge\cdots\wedge g_{\lambda}}(z_0)\ge \max\{\lambda -t-l+1,0\}.$$
Combining this inequality with (\ref{4.2}) we get
\begin{align*}
d\nu_{\tilde{f_1}\wedge\cdots\wedge\tilde{f_{\lambda}}}(z_0)&\ge\sum_{i=1}^td\min_{1\le v\le \lambda}\nu_{(f_v,a_{j_i})}(z_0)+d\max\{\lambda -t-l+1,0\}\\
&\ge\sum_{i=1}^td(\min_{1\le v\le \lambda}\nu_{(f_v,a_{j_i})}(z_0)-\min\{1,\nu_{(f_1,a_{j_i})}(z_0)\})+d(\lambda-l+1)\\
&\ge\sum_{i=1}^\lambda d(\min_{1\le v\le \lambda}\nu_{(f_v,a_{j_i})}(z_0)-\min\{1,\nu_{(f_1,a_{j_i})}(z_0)\})\\
&\ \ \ \ +\sum_{i=1}^q(\lambda-l+1)\min\{1,\nu_{f_1,a_i}(z_0)\}\\
&\ge\sum_{i=1}^\lambda d(\min_{1\le v\le \lambda}\nu_{(f_v,a_{j_i})}(z_0)-\min\{1,\nu_{(f_1,a_{j_i})}(z_0)\})\\
&\ \ \ \ +\sum_{i=1}^q(\lambda-l_i+1)\min\{1,\nu_{f_1,a_i}(z_0)\}
\end{align*}
Hence, the claim is proved.

We now continue to prove Theorem \ref{1.2}. For each $1\leq j\leq q,$ we set 
$$ N_j(r)=\sum_{i=1}^qN^{[k]}(r,\nu_{(f_i,a_j)})-((\lambda-1)k-1)N^{[1]}(r, \nu_{(f_1,a_j)}).$$
For each permutation $J=(j_1,\ldots,j_q)$ of $(1,\ldots,q)$, we set 
$$T_J=\{r\in[1,+\infty):N_{j_1}(r)\geq\cdots\geq N_{j_q}(r)\}.$$ It is clear that $\bigcup_JT_J=[1,+\infty)$. Thefore, there exists a permutation, for instance it is $J_0=(1,\cdots,q)$, such that $\int_{T_{J_0}}dr=+\infty.$ Then we have 
$$ N_{1}(r)\geq\cdots\geq N_{q}(r),$$
for all $r\in T_{J_0}.$ By the assumption that $f_1\wedge\cdots\wedge f_{\lambda}\not\equiv0,$ there exists ordered set of indices $I=\{i_1,\cdots,i_{\lambda}\}$ with $1= i_1< \cdots < i_{\lambda}\leq n$ such that $\det B_I\not\equiv 0$. We note that 
$$N_1(r)=N_{i_1}(r)\geq\cdots\geq N_{i_{\lambda}}(r)\geq N_{n+1}(r),$$ for each $r\in T_{J_0}$.

We see that $\min_{1\leq i\leq \lambda}m_i\geq\sum_{i=1}^{\lambda}\min\{k,m_i\}-(\lambda-1)k$ for every non-negative integers $m_1,\cdots,m_{\lambda}.$ Then by Claim \ref{4.2}, we have 
\begin{align*}
d\sum_{i\in I}(\sum_{v=1}^{\lambda}\min\{\nu_{(f_v,a_i)}(z),k\}&-((\lambda-1)k-1)\min\{1,\nu_{(f_1,a_i)}\}(z))\\
&+\sum_{i=1}^{q}(\lambda-l_i+1)\min\{1,\nu_{(f_1,a_i)}(z)\}\leq d\nu_{\tilde{f_1}\wedge\cdots\wedge\tilde{f_{\lambda}}}(z)
\end{align*}
for all $z\in\C^m\setminus(\mathcal S\cup\bigcup_{i=1}^{\lambda}I(f_i)\cup(a_{i_1}\wedge\cdots\wedge a_{j_{\lambda}}))^{-1}(0)$ with $||z||>r_0$.
Integrating both sides of the above inequality, for every $r>r_0$ we have 
\begin{align*}
\sum_{i\in I}d\big(\sum_{v=1}^{\lambda}N^{[k]}(r,\nu_{(f_v,a_i)})-((\lambda-1)k&+1)N^{[1]}(r,\nu_{(f_1,a_i)})\big)+ \sum_{i=1}^{q}(\lambda -l_i+1)N^{[1]}(r,\nu_{(f_1,a_i)})\\&
 \le dN_{\tilde f_1\wedge\cdots\wedge \tilde f_{\lambda}}(r)\leq d\sum_{v=1}^{\lambda}T_{f_v}(r)+o(\max_{1\leq v\leq\lambda}\{T_{f_v}(r)\}).
\end{align*}
We set $T(r)=\sum_{v=1}^{\lambda}T(r,f_v).$ Then, for all $r\in J_0, r>r_0,$ we have 
\begin{align*}
\bigg|\bigg|\ dT(r)&\geq d\sum_{j=1}^{\lambda}N_{i_j}(r)+\sum_{i=1}^{q}(\lambda -l_i+1)N^{[1]}(r,\nu_{(f_1,a_i)})+o(\max_{1\leq i\leq\lambda}\{T_{f_i}(r)\})\\
&\geq \frac{d\lambda}{q}\sum_{j=1}^{q}N_{j}(r)+\sum_{i=1}^{q}(\lambda -l_i+1)N^{[1]}(r,\nu_{(f_1,a_i)})+o(\max_{1\leq i\leq\lambda}\{T_{f_i}(r)\})\\
&=\sum_{i=1}^{q}\left(\lambda -l_i+1-\frac{d\lambda((\lambda-1)k+1)}{q}\right)N^{[1]}(r,\nu_{(f_1,a_i)})+\frac{d\lambda}{q}\sum_{j=1}^{q}\sum_{i=1}^{\lambda}N^{[k]}(r,\nu_{(f_i,a_j)})\\
&\ \ \ \ \ \ \ \ \ \ \ \ \ \ \ \ \ \ \ \ \ \ \ \ \ \ \ \ \ \ \ \ \ \ \ \ \ \ \ \ \ \ \ \ \ \ \ \ \ \ \ \ \ \ \ \ \ \ \ \ \ \ \ \ \ \ +o(\max_{1\leq i\leq\lambda}\{T_{f_i}(r)\})\\
&\geq\sum_{j=1}^{\lambda}\sum_{i=1}^{q}\left(\frac{d\lambda}q+\frac{\lambda -l_i+1}{\lambda k}-\frac{d((\lambda-1)k+1)}{qk}\right)N^{[k]}(r,\nu_{(f_j,a_i)})+o(\max_{1\leq i\leq\lambda}\{T_{f_i}(r)\})\\
&\geq\sum_{j=1}^{\lambda}\sum_{i=1}^{q}\frac{q(\lambda-l_i+1)+d\lambda(k-1)}{q\lambda k}N^{[k]}(r,\nu_{(f_j,a_i)})+o(\max_{1\leq i\leq\lambda}\{T_{f_i}(r)\}).\\
\end{align*} 
Therefore, for each $r\in T_{J_0}, r>r_0$, we get
\begin{align}\label{4.4}
\bigg|\bigg|\ {dq\lambda k}T(r)\geq\sum_{j=1}^{\lambda}\sum_{i=1}^{q}({q(\lambda-l_i+1)+d\lambda(k-1)})N^{[k]}(r,\nu_{(f_j,a_i)})+o(\max_{1\leq i\leq\lambda}\{T_{f_i}(r)\}).
\end{align} 

For each $1\le j\le q$, put $\lambda_j=q(\lambda-l_j+1)+d\lambda(k-1)$. We see that 
\begin{align*}
\frac{\sum_{i=1}^q\lambda_i}{\lambda_j}=\frac{\sum_{i=1}^qq(\lambda-l_i+1)+dq\lambda(k-1)}{q(\lambda-l_j+1)+d\lambda(k-1)}\ge\frac{q^2+dq\lambda(k-1)}{q(\lambda-1)+d\lambda(k-1)}\ge 2n-k+2.
\end{align*}  
Hence, applying the Theorem \ref{1.1},  for a real number $\eta\in [\max_{1\le i\le q}\lambda_i,\frac{\sum_{i=1}^q\lambda_i}{2n-k+2}]$, which will be chosen later, we have
\begin{equation*}
\begin{aligned}
\bigg|\bigg|\ \ &\frac{\sum_{j=1}^q{{q(\lambda-l_j+1)}}+dq\lambda(k-1)-(n-k)\eta}{n+2}T_{f_t}(r)\\
&\ \ \ \ \ \ \ \ \ \ \ \  \ \ \ \ \ \ \ \ \ \ \ \ \ \ \ \ \ \ \ \ \ \ \  \leq \sum_{j=1}^q({q(\lambda-l_j+1)+d\lambda(k-1)})N^{[k]}_{(f_t,a_{j})}(r)+o(\max\limits_{1\leq i\leq \lambda}T_{f_i}(r)).\\
\end{aligned}
\end{equation*}  This inequality and (\ref{4.4}) imply that 
\begin{equation*}
\begin{aligned}
\bigg|\bigg|\ \sum_{t=1}^{\lambda}\frac{q^2(\lambda+1)+dq\lambda(k-1)-q\sum_{j=1}^ql_j-(n-k)\eta}{n+2}T_{f_i}(r)&\leq dq\lambda kT(r)+o(\max\limits_{1\leq i\leq q}T_{a_i}(r)).
\end{aligned}
\end{equation*} 
Letting $r\to+\infty, r\in T_{J_0}$, we get  
$$ q^2(\lambda+1)+dq\lambda(k-1)-q\sum_{j=1}^ql_j-(n-k)\eta\le (n+2)dq\lambda k. $$
This implies that
\begin{align}\label{4.5}
q\le\frac{1}{\lambda+1}\left ((n+2)d\lambda k-d\lambda(k-1)+\sum_{j=1}^ql_j+(n-k)\frac{\eta}{q} \right )
\end{align}
Now we choose 
$$\eta =\frac{\sum_{j=1}^q\lambda_j}{2n-k+2}=\frac{q(q(\lambda +1)+d\lambda(k-1)-\sum_{j=1}^ql_j)}{2n-k+2}.$$
By simple computation, from (\ref{4.5}) we easily get that
$$ \frac{(n+2)q}{2n-k+2}\le \frac{1}{\lambda+1}\left ((n+2)d\lambda k-\frac{n+2}{2n-k+2}(d\lambda(k-1)-\sum_{j=1}^ql_j) \right ),$$
$$ \text{i.e., }q\le \frac{1}{\lambda+1}\left (d\lambda k(2n-k+2)-d\lambda(k-1)+\sum_{j=1}^ql_j\right).$$
This is a contradiction. 

Hence, the family $\{f_1,...,f_{\lambda}\}$ is algebraically dependent, i.e., $f_1\wedge\cdots\wedge f_{\lambda}\equiv0$. 
\end{proof}

\noindent{\bf Acknowledgements.} The authors would like to thank the referee for his/her helpful comments on the first version of this paper. We would also like to thank professor Si Duc Quang for his kindly giving us the very recent preprint \cite{Q18} and giving us many heplful suggestion to revise and improve the first version of this paper to the current version. 

\end{document}